\documentclass[graybox]{svmult}

\usepackage{mathptmx}       
\usepackage{helvet}         
\usepackage{courier}        
\usepackage{type1cm}        
\usepackage{makeidx}         
\usepackage{graphicx}        
\usepackage{multicol}        
\usepackage[bottom]{footmisc}

\usepackage{amssymb}
\usepackage{dsfont} 

\newcommand{\be}{\begin{equation}}
\newcommand{\ee}{\end{equation}}
\newcommand{\bel}[1]{\begin{equation}\label{#1}}
\newcommand{\bea}{\begin{eqnarray}}
\newcommand{\eea}{\end{eqnarray}}
\newcommand{\balign}{\begin{align}}
\newcommand{\ealign}{\end{align}}
\newcommand{\ba}{\begin{array}}
\newcommand{\ea}{\end{array}}
\newcommand{\bfig}{\begin{figure}}
\newcommand{\efig}{\end{figure}}
\newcommand{\eref}[1]{(\ref{#1})}

\newcommand{\bra}[1]{\mbox{$\langle \, {#1}\, |$}}
\newcommand{\ket}[1]{\mbox{$| \, {#1}\, \rangle$}}
\newcommand{\exval}[1]{\mbox{$\langle \, {#1}\, \rangle$}}
\newcommand{\inprod}[2]{\mbox{$\langle \, {#1} \, | \, {#2} \, \rangle$}}
\newcommand{\Prob}[1]{\mbox{${\rm Prob}\left[ \, {#1}\, \right]$}}
\newcommand{\comm}[2]{\mbox{$[\,{#1}\,,\,{#2}\,]$}}
\newcommand{\half}{\frac{1}{2}}

\newcommand{\bfx}{\vec{x}}
\newcommand{\bfy}{\vec{y}}

\newcommand{\rme}{\mathrm{e}}
\newcommand{\rmd}{\mathrm{d}}

\newcommand{\C}{{\mathbb C}}
\newcommand{\R}{{\mathbb R}}

\begin{document}

\title*{Duality relations for the ASEP conditioned on a low current}
\author{G.M.~Sch\"utz}
\institute{
G.M.~Sch\"utz \at Institute of Complex Systems II,
Forschungszentrum J\"ulich, 52425 J\"ulich, Germany \email{g.schuetz@fz-juelich.de}}

\maketitle

\abstract{We consider the asymmetric simple exclusion process (ASEP) on a finite lattice
with periodic boundary conditions, conditioned to carry an atypically low current. 
For an infinite discrete set of currents, parametrized by the driving strength $s_K$, 
$K \geq 1$, we prove duality 
relations which arise from the quantum algebra $U_q[\mathfrak{gl}(2)]$ symmetry of the generator of the process with reflecting boundary conditions.
Using these duality relations we prove on microscopic level 
a travelling-wave property of the conditioned 
process for a family of shock-antishock measures for $N>K$
particles: If the initial measure is a member of this family with $K$ microscopic 
shocks at positions
$(x_1,\dots,x_K)$, then the measure at any time $t>0$ of the process with driving strength $s_K$ is a convex combination of such measures with shocks at positions $(y_1,\dots,y_K)$. which can be expressed in terms
of $K$-particle transition probabilities of the conditioned ASEP with driving strength $s_N$.
}

\section{Introduction}
\label{sec:1}

In the asymmetric simple exclusion process (ASEP) \cite{Spit70,Ligg85,Ligg99,Schu01}
each lattice site $k$ on a lattice $\Lambda = (1,\dots, L)$ is occupied by at most 
one particle, indicated by occupation numbers $\eta(k) \in \mathbb{S} = \{0,1\}$. 
We denote by $\eta = (\eta(1), \dots, \eta(L))\in \Omega = \mathbb{S}^L$
a configuration $\eta$ of the particle system.
Informally speaking, in one dimension particles  try to 
jump to the right 
with rate $r = wq$ and to the left 
with rate $\ell = wq^{-1}$. The jump attempt is successful if the target site is empty,
otherwise the jump attempt is rejected. 
The invariant measures of the ASEP 
with periodic boundary conditions are well-known:
For fixed particle number $N$ these are the uniform measures.
From these one can construct the grandcanonical Bernoulli 
product measures with fugacity $z=\rho/(1-\rho)$ where $\rho=N/L$ is the particle density 
on the torus. For these measures, where each lattice site $k$ is occupied with probability
$\rho$ independently of all other sites, one has a stationary particle
current $j^\ast = (r-\ell) \rho(1-\rho)$, corresponding to
an expected mean time-integrated current $\exval{J(t)}/t = j^\ast$.

In the context of macroscopic fluctuation theory \cite{Bert15} one is interested in 
conditioning the process on fluctuations around 
some atypical mean time-integrated current $j\neq j^\ast$. 
A question of fundamental interest is then 
which macroscopic density profile is most
likely to realize such a large deviation of the current inside a very large 
(more precisely: infinite) time interval of conditioning. 
This large-deviation problem thus concerns an untypical
ensemble of trajectories of the process. This ensemble is usually not defined by
$j$ but
via Legendre transformation in terms of the canonically
conjugate driving strength $s(j)$ with $s(j^\ast)=0$.
Interestingly,
for conditioning on a lower-than-typical current (i.e., for $s<0$),
it was found by Bodineau and Derrida
\cite{Bodi05} for the weakly asymmetric simple exclusion process 
that there is a dynamical phase transition: For currents slightly below the
typical value  $j^\ast$ the optimal macroscopic profile is constant as it is 
for $j^\ast$. However,
below a critical threshold $j_c < j^\ast$ (corresponding to some $s_c<0$) 
the optimal macroscopic profile 
is a travelling wave with a shape resembling a smoothened shock/antishock pair.

More recently, in a similar setting, but for finite duration $t$ of conditioning, 
the microscopic structure of a travelling wave in the ASEP (not weakly!)
was elucidated in detail for a specific choice of negative driving strength \cite{Beli13}: 
One considers a certain family of
inhomogeneous product measures $\mu_{k}$, 
indexed by a lattice site $k$, where the microscopic
density profile as function of the position on the
lattice has a density-jump 
at position $k$ on the torus, analogous to a shock on macroscopic scale. 
At time $t=0$
$N$ particles ($N$ arbitrary) 
are distributed according to
the restricted
measure $\mu_{k}^N \propto \mu_k \delta_{\sum_k \eta(k), N}$. Then at any future time $t>0$ 
of the conditioned dynamics the measure is a convex combination $\mu_{k}^N(t) = 
\sum_l c(l,t|k,0) \mu_{l}^N$
of such measures. 
The weights $c(l,t|k,0)$ are the transition probabilities of a single biased random walk, 
thus suggesting that a shock in a macroscopic travelling wave performs a biased 
random walk on 
microscopic level.

In this work we trace back the mathematical origin of this rigorous result
to certain algebraic properties of the generator of the process. Then, using these 
properties, we go beyond \cite{Beli13} to derive a family of duality relations 
that allow us 
to construct more complex microscopic structures corresponding
to more general macroscopic optimal profiles. The starting point is the well-known fact that
for reflecting boundaries, where the process is reversible, 
the generator of the process commutes with the generators 
of the quantum algebra $U_q[\mathfrak{gl}(2)]$ \cite{Alca93}. This fact has been 
used in \cite{Sand94} to construct the canonical reversible measures and in 
\cite{Schu97} to derive self-duality relations for the unconditioned ASEP. 
\footnote{We mention that the deep link between 
duality of Markov processes and symmetries of its generator, 
first noted in \cite{Schu94}, 
that we exploit here 
was given a systematic abstract treatment in \cite{Giar09}. More recently many concrete
symmetry-based dualities for interacting particle systems were derived using this approach
\cite{Ohku10,Cari13,Cari14,Cari15,Boro14,Corw15,Beli15b,Kuan15,Cari15b}.}

On the torus, however, the symmetry of the generator under $U_q[\mathfrak{gl}(2)]$ breaks 
down. Nevertheless, some time ago Pasquier and Saleur \cite{Pasq90} found 
intertwining relations involving the generators of 
$U_q[\mathfrak{gl}(2)]$ and the Heisenberg quantum Hamiltonian with a 
boundary twist. This quantum Hamiltonionan operator became later to be known to be
closely related to the generator of the conditioned ASEP \cite{Derr98}.  
Here we present a new proof for the
results of \cite{Schu97} (and correct some typos there) for reflecting boundaries
and make use of intertwining relations 
of \cite{Pasq90} (correcting some typos also in that paper) to derive 
an infinite discrete family of duality relations
for the ASEP with {\it periodic} boundary conditions. 
These new duality relations apply to the 
process conditioned on fluctuations around some untypically low mean 
time-integrated current. 

The simplest of these duality relations
proves that the homogeneous Bernoulli measure is the
invariant measure for the unconditioned ASEP with periodic boundary conditions. 
The derivation of this well-known fact
from the $U_q[\mathfrak{gl}(2)]$-symmetry 
of the process
with reflecting boundary questions is remarkable in so far as it raises the interesting question whether one
can construct the matrix product measures \cite{Ferr07,Evan09} of the periodic 
multi-species ASEP 
from the $U_q[\mathfrak{gl}(n)]$-symmetry of that process with reflecting boundaries\cite{Alca93,Beli15a}.

From the non-trivial higher order duality relations we obtain an infinite discrete family of
new microscopic ``travelling waves''
for the conditioned process.

\section{Definitions and notation}
\label{Sec:Definotat}

It is convenient to work with the quantum Hamiltonian formalism \cite{Lloy96,Schu01} 
where the generator of the process is represented by a matrix which in a judiciously chosen basis turns out to be closely related to the Hamiltonian operator of a physical 
quantum  system. We first introduce some notation and then describe
in some detail the tools required for the quantum Hamiltonian formalism
for the benefit of readers not familiar with this approach.

\subsection{State space and configurations}
\label{Sec:Defconfig}

We say that
a site $k \in \Lambda$ is occupied by a particle if $\eta(k) = 1$ or that
it is empty if $\eta(k) = 0$. 
The fact that a site can be occupied by at most one particle is the exclusion principle.
Occasionally we denote configurations with a fixed number of $N$ particles
by $\eta_{N}$. 
The set of all configurations with $N$ particles is denoted $\Omega_N$.
We also define
\bel{lonv}
\upsilon(k) := 1-\eta(k)
\ee
and the particle numbers
\bel{partnum}
N(\eta) = \sum_{k=1}^L \eta(k), \quad V(\eta) = \sum_{k=1}^L \upsilon(k) = L-N.
\ee

A useful alternative way of presenting uniquely of configuration $\eta_N$ is
obtained by labelling the particles 
consecutively from left to right (clockwise) by 1 to $N$ and their positions
on $\Lambda$ by $x_i \ \mbox{mod} \ L$.
A configuration $\eta$ is then represented by the set
$\bfx := \{ x \, :\, \eta(x) = 1 \}$.
We call this notation the position representation. 
We shall use interchangeably the arguments $\eta$, $\bfx$, 
for functions of the configurations. When the argument is clear from context
it may be omitted.
We note the trivial, but frequently used identities $N(\eta)  \equiv  N(\bfx) = |\bfx|$
and
\be
\label{occupos}
\eta(k) = \sum_{i=1}^{N(\bfx)} \delta_{x_i,k}.
\ee
For a configuration $\eta\equiv \bfx$ we also define 
the number $N_k(\eta)$ of 
$A$-particles to the left of a particle
at site $k$
\bel{NyMx}
N_k(\eta) := \sum_{i=1}^{k-1} \eta(i) = \sum_{i=1}^{N(\eta)} \sum_{l=1}^{k-1} \delta_{x_i,l}.
\ee

Furthermore, for $1\leq k \leq L-1$ we define the local permutation
\bel{etaexch}
\pi^{kk+1}(\eta) = \{ \eta(1), \dots \eta(k-1), \eta(k+1), \eta(k), \eta(k+2), \dots ,\eta(L) \}
=: \eta^{kk+1} ,
\ee
and for $k=L$ we define 
\bel{etaexchL}
\pi^{L1}(\eta) = \{ \eta(L), \dots , \eta(k),  \dots ,\eta(1) \}
=: \eta^{L1} .
\ee
The space reflection is defined by
\bel{etarefl}
R(\eta) = \{ \eta(L), \eta(L-1), \dots , \eta(1)\}
\ee
corresponding to $R(\eta(k)) = \eta(L+1-k)$ for the occupation numbers.

\subsection{Definition of the ASEP}
\label{Sec:Defprocess}

For functions $f:\mathbb{S}^L \to \C$ the ASEP $\eta_t$ with periodic boundary conditions
and hopping asymmetry $q$ is defined by the generator
\bel{generator}
\mathcal{L} f(\eta) := {\sum}_{\eta'\in\mathbb{S}^L}' w(\eta\to\eta') [f(\eta') - f(\eta)]
\ee
where 
the transition rates between configurations
\bel{rates}
w(\eta\to\eta') = \sum_{k=1}^{L} w^{kk+1}(\eta) \delta_{\eta',\eta^{kk+1}}
\ee
are defined in terms of the local hopping rates
\bel{localrates}
w^{kk+1}(\eta) = w \left[ q  \eta(k) \upsilon(k+1) +
q^{-1}  \upsilon(k) \eta(k+1) \right].
\ee
The prime at the summation symbol \eref{generator} indicates the
absence of the term $\eta' = \eta$ which is omitted since $w(\eta\to\eta)$ is not
defined.\footnote{When the summation is over $\Omega=\mathbb{S}^L$ we shall usually 
omit the set $\mathbb{S}^L$ under the summation symbol and simply write $\sum_{\eta}$.} 
The transition rates are non-zero only for a transition from a configuration 
$\eta$ to a configuration $\eta'=\eta^{kk+1}$ defined by \eref{etaexch}. 

We shall assume partially asymmetric hopping $q\neq 0,1,\infty$. 
The constant $w \neq 0$ sets the time scale of the process. 
On the torus we identify increasing order of the lattice index with the clockwise direction.
In the case of reflecting boundary conditions no jumps from site 1 to the left and no jumps from site $L$ to the right are allowed. 
Increasing order of the lattice index is identified with the direction left to right. 
The upper summation limit $L$ in \eref{rates}
has to be replaced by $L-1$, giving rise to a generator that we denote by
$\tilde{\mathcal{L}}$. 

In order study fluctuations around
some untypical integrated current,
parameterized in terms of the driving strength $s$,
we define the weighted transition rates
\bea
\label{wlocalratesbulk}
w_s^{kk+1}(\eta) & = & w \left[ q \rme^{s} \eta(k) \upsilon(k+1) +
q^{-1} \rme^{-s} \upsilon(k) \eta(k+1) \right], \quad 1\leq k \leq L-1 \\
\label{wlocalratesbound}
w_{s,\bar{s}}^{L1}(\eta) & = & w \left[ q \rme^{s+\bar{s}}  \eta(L) \upsilon(1) +
q^{-1} \rme^{-s-\bar{s}} \upsilon(L) \eta(1) \right], \quad k=L.
\eea
This leads us to define
the weighted generators
\bea
\label{wgeneratorrefl}
\tilde{\mathcal{L}}_s f(\eta) & := & \sum_{k=1}^{L-1} w_s^{kk+1}(\eta) f(\eta^{kk+1}) 
- w^{kk+1}(\eta)f(\eta) \\
\label{wgeneratorper}
\mathcal{L}_{s,\bar{s}} f(\eta)
 & := & \tilde{\mathcal{L}}_s f(\eta) 
 + w_{s,\bar{s}}^{L1}(\eta) f(\eta^{L1}) - w^{L1}(\eta)f(\eta).
\eea
The weighted generators give a weight $\rme^{s}$ ($\rme^{-s}$) to each particle jump to the
right (left) anywhere on the lattice and for the process with periodic boundary conditions 
an extra weight $\rme^{\bar{s}}$ ($\rme^{-\bar{s}}$) to each particle 
jump to the
right (left) across bond $(L,1)$.
Thus each random trajectory 
of the process is given a weight
$\rme^{sJ(t)+\bar{s}J_L(t)}$ where $J(t)$ is the time-integrated total 
current, i.e., the total number
of all particle jumps to the right up to time $t$ minus the total number
of all particle jumps to the left up to time $t$ and $J_L(t)$ is the 
time-integrated  current across bond $(L,1)$. Notice that the diagonal part
of the weighted generator does not depend on the driving strength 
$s$ or $\bar{s}$, 
reflecting the fact that the random times
after which jumps occur remain unchanged. 
For details on this construction see e.g. \cite{Harr07,Jack10,Chet14} 
and specifically for the present context \cite{Schu15}.

We fix more notation and summarize some well-known basic facts from the theory of Markov processes.
For a probability distribution $P(\eta)$ we denote the expectation of a continuous function 
$f(\eta)$ by $\exval{f}_P := \sum_{\eta} f(\eta) P(\eta)$.
The transposed generator is defined by
$\mathcal{L}^T f(\eta) := {\sum}'_{\eta' \in\mathbb{S}^L} f(\eta') \mathcal{L} \, \mathbf{1}_{\eta'}(\eta)$
where $\mathbf{1}_{\eta'}(\eta) = \delta_{\eta,\eta'}$.
With this definition \eref{generator} yields for a probability distribution $P(\eta)$ the
{\it master equation} 
\bel{transgenerator}
\mathcal{L}^T P(\eta) = 
{\sum}'_{\eta'} [ w(\eta'\to\eta) P(\eta') - w(\eta\to\eta') P(\eta)].
\ee

An invariant measure is denoted $\pi^\ast(\eta)$ and defined by 
\be
\label{invmeasure1}
\mathcal{L}^T \pi^\ast(\eta) = 0
\ee
and the normalization $
\sum_{\eta} \pi^\ast(\eta) = 1$.
An unnormalized measure with the property \eref{invmeasure1} is 
denoted $\pi(\eta)$.
The time-reversed process is defined by 
\bel{revgenerator}
\mathcal{L}^{rev} f(\eta) := {\sum_{\eta'}}' w^{rev}(\eta\to\eta') [f(\eta') - f(\eta)]
\ee
with 
$w^{rev}(\eta\to\eta') = w(\eta'\to\eta) \pi(\eta') / \pi(\eta)$.
The process is reversible if $\mathcal{L}^{rev} = \mathcal{L}$ which means that
the rates satisfy the detailed balance condition
$\pi(\eta) w(\eta\to\eta') = w(\eta'\to\eta) \pi(\eta')$. 
A probability distribution satisfying the detailed balance condition is a 
reversible measure. It is easily verified that the ASEP with reflecting boundary
conditions is reversible with reversible measure
\bel{ASEPrevmeas}
\pi(\eta) = q^{\sum_{k=1}^{L} (2k+\mu) \eta(k)}
= \rme^{\mu N(\eta)} q^{2\sum_{i=1}^{N(\eta)} x_i} .
\ee
for any $\mu\in\R$.

We define the transition matrix $H$ of the process by the matrix elements
\bel{transmatrix}
H_{\eta'\eta} = \left\{ \ba{ll} 
- w(\eta\to\eta') \quad & \eta \neq \eta' \\
{\sum}'_{\eta'} w(\eta\to\eta') & \eta = \eta' .
\ea \right.
\ee
with $w(\eta\to\eta')$ given by \eref{rates}. One has
\bel{generator2}
\mathcal{L} f(\eta) = - \sum_{\eta'} f(\eta') H_{\eta'\eta}, \quad
\mathcal{L}^T P(\eta) = - \sum_{\eta'} H_{\eta\eta'} P(\eta').
\ee
Notice that here the sum includes the term $\eta'=\eta$. 
In slight abuse of language we shall
also call $H$ the generator of the process.
Analogously we also define the weighted transition matrix 
where the off-diagonal
elements are replaced by the weighted rates \eref{wlocalratesbulk},
\eref{wlocalratesbound}. 

For an unnormalized stationary distribution we define the diagonal matrix
$\hat{\pi}$ with the stationary weights $\pi(\eta)$ on the diagonal.
For ergodic processes with finite state space one has $0< \pi(\eta) <\infty$ for all $\eta$.
In terms of this diagonal matrix we can write the generator of the reversed dynamics as 
$H^{rev}  = \hat{\pi} H^T \hat{\pi}^{-1}$.
The reversibility condition $H^{rev}=H$ then reads
\bel{revASEPrefl}
\hat{\pi}^{-1} H \hat{\pi} = H^T.
\ee
Therefore, if one finds a diagonal matrix with the property \eref{revASEPrefl}
then this matrix defines a reversible measure.

\subsection{Representation of the generator in the natural tensor basis}

In order to write the matrix $H$ explicitly we assign to
each configuration $\eta$ a canonical basis vector $\ket{\eta}$.
We choose the binary ordering $\iota(\eta) = 1+ \sum_{k=1}^L \eta(k) 2^{k-1}$
of the basis. Defining single-site basis vectors of dimension 2 
\be
|0) := \left( \ba{c} 1 \\ 0 \ea \right) , \quad |1) := \left( \ba{c} 0 \\ 1 \ea \right)
\ee
one then has $\ket{\eta} = |\eta(1)) \otimes \dots \otimes |\eta(L))$
where $\otimes$ denotes the tensor product. These basis vectors span the complex
vector space $(\C^2)^{\otimes L}$ of dimension $d=2^L$. We also define transposed basis vectors
$\bra{\eta} := \ket{\eta}^T$ and the inner product 
$\inprod{v}{w} := \sum_{\eta} v(\eta) w(\eta)$.

Furthermore we define the two-by-two Pauli matrices
\bel{Pauli}
\sigma^x := \left( \ba{cc} 0 & 1 \\ 1 & 0 \ea \right), \quad
\sigma^y := \left( \ba{cc} 0 & -i \\ i & 0 \ea \right), \quad
\sigma^z := \left( \ba{cc} 1 & 0 \\ 0 & -1 \ea \right)
\ee
and the two-dimensional unit matrix $\mathds{1}$. From these we construct
\bel{ladder}
\sigma^\pm = \half (\sigma^x \pm i \sigma^y), \quad \hat{n} = \half (\mathds{1} - \sigma^z), \quad \hat{\upsilon} = \half (\mathds{1} + \sigma^z).
\ee
These matrices satisfy the following relations:
\bel{projhat}
\ba{llll} 
\displaystyle
\sigma^+ \sigma^- = \hat{\upsilon}, \quad & \sigma^- \sigma^+ = \hat{n}, \quad &
\hat{n} \hat{\upsilon} = 0, \quad & \hat{\upsilon} \hat{n} = 0 \\
\sigma^+ \hat{n} = \sigma^+, & \hat{n} \sigma^+ = 0, &
\sigma^+ \hat{\upsilon} = 0, & \hat{\upsilon} \sigma^+ =  \sigma^+, \\
\sigma^- \hat{n} = 0, & \hat{n} \sigma^- = \sigma^-, & 
\sigma^- \hat{\upsilon} = \sigma^-, & \hat{\upsilon} \sigma^- = 0.
\ea
\ee
With the occupation variables \eref{lonv} for a single site we have the projector property
\bel{proj}
\hat{n} |\eta) = \eta |\eta), \quad \hat{\upsilon} |\eta) = \upsilon |\eta).
\ee
Having in mind the action of these operators to the right on a column vector, 
we call $\sigma^-$ a creation operator, and $\sigma^+$ 
annihilation operators. When acting to the left on a bra-vector the roles are
interchanged: $\sigma^+$ acts as creation operator and $\sigma^-$ as
annihilation operator.

For $L>1$ and any linear combination $u$ of these matrices we define the tensor operators 
$u_k := \mathds{1}^{\otimes k-1} \otimes u \otimes \mathds{1}^{\otimes L-k}$.
By convention the zero'th tensor power of any matrix is the $c$-number 1
and $u^{\otimes 1} = u$. We note that also
the tensor occupation operators $\hat{n}_k$ act as projectors
\be
\label{projabL}
\hat{n}_k \ket{\eta} = \eta(k) \ket{\eta} = \sum_{i=1}^{N(\eta)} \delta_{x_i,k} \ket{\eta}, 
\ee
with the occupation variables $\eta(k)$ or particle coordinates $x_i$ respectively
understood as functions of $\eta$.
The proof is trivial: The first equality is inherited from \eref{proj} by 
multilinearity of the tensor product, the
second equality follows from \eref{occupos}.
Multilinearity of the tensor product also yields
$u_k v_{k+1} = \mathds{1}^{\otimes (k-1)} \otimes [(u \otimes \mathds{1}) (\mathds{1} \otimes v)]
\otimes \mathds{1}^{\otimes (L-k-1)} = \mathds{1}^{\otimes (k-1)} \otimes ( u \otimes v )
\otimes \mathds{1}^{\otimes (L-k-1)}$ and the commutator property
$u_k v_l = v_l u_k$ for $k\neq l$.
For $k=l$ one has relations analogous to \eref{projhat}.

It turns out to be convenient to introduce parameters $\alpha=q \rme^s$ and 
$\beta = \rme^{\bar{s}}$ and express for periodic boundary conditions the weighted generator as 
$H(q,\alpha,\beta)$ with the convention $H(q,q,1) = H$ for the unweighted generator.
Similarly one writes $\tilde{H}(q,\alpha)$ for the weighted generator with reflecting
boundary conditions with the convention $\tilde{H}(q,q)=\tilde{H}$.
With these definitions the weighted generators $\tilde{H}(q,\alpha)$ 
and $H(q,\alpha,\beta)$ defined by \eref{wgeneratorrefl} and \eref{wgeneratorper} 
resp. become
\bea
\label{ASEPgenrefl}
\tilde{H}(q,\alpha) & = & \sum_{k=1}^{L-1} h_{k,k+1}(q,\alpha) \\
\label{ASEPgenper}
H(q,\alpha,\beta) & = &  \tilde{H}(q,\alpha) + h_{L,1}(q,\alpha,\beta)
\eea
with the hopping matrices
\bel{hoppingbulk}
h_{k,k+1}(q,\alpha) = - w \left[ \alpha \sigma^+_k \sigma^-_{k+1} - 
q \hat{n}_k \hat{\upsilon}_{k+1}
+  \alpha^{-1} \sigma^-_k \sigma^+_{k+1} - q^{-1} \hat{\upsilon}_k \hat{n}_{k+1} \right]
\ee
and
\bel{hoppingbound}
h_{L,1}(q,\alpha,\beta) = - w \left[ \alpha \beta \sigma^+_k \sigma^-_{k+1} - 
q \hat{n}_k \hat{\upsilon}_{k+1}
+  (\alpha \beta)^{-1} \sigma^-_k \sigma^+_{k+1} - q^{-1} \hat{\upsilon}_k \hat{n}_{k+1}) \right].
\ee

It is useful to introduce the space-reflection operator $\hat{R}$ defined by 
\bel{Def:Reflection}
\hat{R} u_{k} \hat{R}^{-1} = u_{L+1-k}
\ee 
for local one-site operators $u_k$ and the diagonal transformations
\bea
\label{Vg}
V(\gamma) & = & \gamma^{\frac{1}{4} \sum_{k=1}^{L} (2k-L-1) \sigma_k^z} = 
\gamma^{-\frac{1}{2} \sum_{k=1}^{L} (2k-L-1) \hat{n}_k} \\
\label{numb}
W(z) & = & z^{\hat{N}}
\eea
with the number operator $\hat{N}=\sum_{K=1}^L \hat{n}_k$.
We note the properties
\bea
\label{Vgsym}
\hat{R} V(\gamma) \hat{R}^{-1} & = & V^{-1}(\gamma) = V(\gamma^{-1}), \\
\label{numbsym}
\hat{R} W(z) \hat{R}^{-1} & = & W(z) = W^{-1}(z^{-1}), \\
\label{symgen1}
\hat{R} H(q,\alpha,\beta) \hat{R}^{-1} & = & H(q,\alpha^{-1},\beta^{-1}) \\
\label{HTper}
H^T(q,\alpha,\beta) & = & H(q,\alpha^{-1},\beta^{-1})  \\
\label{symgen3}
W H(q,\alpha,\beta) W^{-1} & = & H(q,\alpha,\beta).
\eea

Moreover, the transformation property
\bel{Vgsig}
V(\gamma) \sigma_k^\pm V^{-1}(\gamma) = \gamma^{\pm \half (2k-L-1)} \sigma_k^\pm 
\ee
yields
\bea
\label{Hgaugeper}
V(\gamma) H(q,\alpha,\beta) V^{-1}(\gamma) & = & H\left(q,\alpha \gamma^{-1},\beta \gamma^L\right) \\
\label{Hgaugerefl}
V(\gamma) \tilde{H}(q,\alpha) V^{-1}(\gamma) & = & 
\tilde{H}\left(q,\alpha \gamma^{-1}\right).
\eea
Thus for periodic boundary conditions global conditioning and
local conditioning are related by a similarity transformation,
while for reflecting
boundary conditions the conditioning can be completely absorbed into a similarity transformation.
One also finds with $\gamma=q^2$, $\alpha = q$ the reversibility relation
$V(q^2) \tilde{H}(q) V^{-1}(q^2) =
\tilde{H}^T$.
By \eref{revASEPrefl} this shows that 
$\hat{\pi} = V^{-1}(q^2)$
is the matrix form of the reversible measure
\eref{ASEPrevmeas} with $\mu = -L-1$ \cite{Schu97}.\footnote{This is equivalent to 
Eq. (2.14) in \cite{Schu97}, which, however, has a sign error and should read 
$H^T = V^{-2} H V^2$.}

For driving strength
$s_0:=-\ln{(q)}$
corresponding to $\alpha=1$
one finds for $1\leq k \leq L-1$
\be 
h_{k,k+1}(q,1) = 
- \frac{w}{2} \left[\sigma^x_k \sigma^x_{k+1} + \sigma^y_k \sigma^y_{k+1}
+ \Delta (\sigma^z_k \sigma^z_{k+1} - \mathbf{1}) + h (\sigma^z_k - \sigma^z_{k+1}) \right]
\ee
with 
\be 
\Delta = \half (q+q^{-1}), \quad h = \half (q-q^{-1})
\ee
and the unit-matrix $\mathbf{1}$ of dimension $2^L$. Notice that for periodic boundary conditions
the local divergence terms $\sigma^z_k - \sigma^z_{k+1}$ 
cancel. For reflecting boundaries the local divergence term contributes 
opposite boundary fields $h (\sigma^z_L - \sigma^z_1)$.
With the further choice $\bar{s}_0 = 0$ corresponding to
$\beta=1$ the weighted generator \eref{ASEPgenper}
becomes the Hamiltonian operator $H(q,1,1)$
of the ferromagnetic Heisenberg spin-1/2 quantum chain, while for
$\beta \neq 1$ one has the Heisenberg chain 
$H(q,1,\beta)$ with twisted boundary conditions \cite{Pasq90}.

Since particle number is conserved the process is trivially reducible. For each particle
number $N$ one has an irreducible process $\eta_{N,t}$ on the state space $\Omega_N$. 
We define the projector 
\bel{projectorN}
\hat{\mathbf{1}}_N := \sum_{\eta \in \Omega_N} \ket{\eta}\bra{\eta}
\ee
where we have used the quantum mechanical 
ket-bra convention $\ket{\eta}\bra{\eta} \equiv \ket{\eta}\otimes \bra{\eta}$ for the tensor
product of two vectors. Thus one obtains the generator
\bel{generatorN}
H_N(q,\alpha,\beta) := \hat{\mathbf{1}}_N \ H(q,\alpha,\beta) \ \hat{\mathbf{1}}_N
\ee
for the $N$-particle weighted ASEP. 

Notice that $\hat{\mathbf{1}}_N$ acts as
unit matrix on the irreducible subspace corresponding to particle number $N$.
The unit matrix $\mathbf{1}$ in the full space has the useful representation
\bel{unit}
\mathbf{1} = \sum_{\eta \in \Omega} \ket{\eta}\bra{\eta}.
\ee

\subsection{The quantum algebra $U_q[\mathfrak{gl}(2)]$}

The quantum algebra
$U_q[\mathfrak{gl}(2)]$ is the $q$-deformed universal enveloping algebra of 
the Lie algebra $\mathfrak{gl}(2)$. This associative algebra over $\C$ is
generated by $\mathbf{L}_i^{\pm 1}$, $i=1,2$ and $\mathbf{S}^\pm$
with the relations \cite{Jimb86,Burd92}
\bea
\label{Uqglndef1}
& & \comm{\mathbf{L}_i}{\mathbf{L}_j} = 0  \\
\label{Uqglncomm2}
& & \mathbf{L}_i \mathbf{S}^\pm = q^{\pm (\delta_{i,2} - \delta_{i,1})} \mathbf{S}^\pm
 \mathbf{L}_i\\
\label{Uqglncomm3}
& & \comm{\mathbf{S}^+}{\mathbf{S}^-} = 
\frac{(\mathbf{L}_{2}\mathbf{L}_1^{-1})^2 - (\mathbf{L}_{2}\mathbf{L}_1^{-1})^{-2}}{q-q^{-1}}
\eea
Notice the replacement $q^2 \to q$ that we made in the definitions of  \cite{Burd92}.

It is convenient to work also with the subalgebra $U_q[\mathfrak{sl}(2)]$.
We introduce the generators $\mathbf{N}$ and $\mathbf{V}$
via $q^{-\mathbf{N}/2} = \mathbf{L}_1$, $q^{-\mathbf{V}/2} = \mathbf{L}_2$ and define
\bel{complement}
\mathbf{S}^z = \half(\mathbf{N} - \mathbf{V})
\ee
and the identity $I$.
Then the quantum algebra $U_q[\mathfrak{sl}(2)]$ is the subalgebra generated by 
$q^{\pm \mathbf{S}^z}$ and $\mathbf{S}^\pm$ with relations
\bea
\label{Uqslncomm1b}
& & q^{\mathbf{S}^z} q^{-\mathbf{S}^z} = q^{-\mathbf{S}^z} q^{\mathbf{S}^z} = I\\
\label{Uqslncomm2b}
& & q^{\mathbf{S}^z} \mathbf{S}^\pm q^{-\mathbf{S}^z} = q^{\pm 1} \mathbf{S}^\pm\\
\label{Uqslncomm3b}
& & \comm{\mathbf{S}^+}{\mathbf{S}^-} =  
\frac{q^{2\mathbf{S}^z} - q^{-2\mathbf{S}^z}}{q-q^{-1}}
\eea
Observing that 
$\mathbf{N} + \mathbf{V}$ belongs to the center of $U_q[\mathfrak{gl}(2)]$ \cite{Jimb86}
one sees that $U_q[\mathfrak{sl}(2)]$ is a subalgebra of $U_q[\mathfrak{gl}(2)]$.

It is trivial to verify that $U_q[\mathfrak{gl}(2)]$ has the two-dimensional fundamental
representation $\mathbf{S}^\pm \to \sigma^\pm$,  $\mathbf{N} \to \hat{n}$,
$\mathbf{V}  \to \hat{\upsilon}$ given by the matrices \eref{ladder}. 
Then $\sigma^\pm$ and $\sigma^z/2$ form the two-dimensional
fundamental representation of $U_q[\mathfrak{sl}(2)]$. 
Reducible higher-dimensional representations can be constructed using the coproduct
 \cite{Jimb86}
\bea
\label{coprod1}
\Delta(\mathbf{S}^\pm) & = & \mathbf{S}^\pm \otimes q^{-\mathbf{S}^z} + 
q^{\mathbf{S}^z} \otimes \mathbf{S}^\pm \\
\label{coprod2}
\Delta(\mathbf{S}^z) & = & \mathbf{S}^z \otimes \mathds{1} + \mathds{1} \otimes \mathbf{S}^z.
\eea
By repeatedly applying the coproduct to the fundamental representation
we obtain 
\bea
S^\pm(k) & = & q^{\half \sum_{j=1}^{k-1} \sigma^z_j - \half \sum_{j=k+1}^{L} \sigma^z_j} \sigma^\pm_k \\
S^z(k) & = & \half \sigma^z_k .
\eea
One has 
\bea
S^\pm(k) S^\pm(l) & = & \left\{ \ba{cl} 
q^{\pm 2} S^\pm(l) S^\pm(k) & \quad k>l \\ 
0 & \quad k=l \\ q^{\mp 2} S^\pm(l) S^\pm(k) & \quad k<l 
\ea \right.
\eea
Thus the spatial order in which particles are created (or annihilated)
by applying the operators $S^\pm(k)$ gives rise to
combinatorial issues when building many-particle configurations 
from the reference state corresponding to the empty lattice.

From the coproduct one obtains
the tensor representations of $U_q[\mathfrak{sl}(2)]$, denoted by capital letters,
\bel{repX}
S^\pm = \sum_{k=1}^L  S^\pm(k), \quad S^z = \sum_{k=1}^L  S^z(k).
\ee
For the full quantum algebra $U_q[\mathfrak{gl}(2)]$ the tensor generators are $S^\pm$
and
$\hat{N} = \sum_{k=1}^L \hat{n}_k$, 
$\hat{V} = \sum_{k=1}^L \hat{\upsilon}_k$.
The unit $I$ is represented by the $2^L$-dimensional unit matrix
$\mathbf{1} := \mathds{1}^{\otimes L}$.

For reflecting boundary conditions the Heisenberg Hamiltonian $\tilde{H}(q,1)$ 
is symmetric under the action of $U_q[\mathfrak{gl}(2)]$ \cite{Alca93,Pasq90}.
This symmetry property is the origin of the duality relations derived in \cite{Schu97}
and will also be used extensively below.
In fact, for $1\leq k \leq L-1$ one has
$\comm{h_{k,k+1}(q,1)}{S^\pm} = \comm{h_{k,k+1}(q,\alpha)}{S^z} = 0$,
which imply 
\bel{symmetryH1}
\comm{\tilde{H}(q,1)}{S^\pm} = 0 
\ee
and, equivalently to \eref{symgen3}, the diagonal symmetries
$\comm{\tilde{H}(q,\alpha)}{\hat{N}} = \comm{\tilde{H}(q,\alpha)}{\hat{V}} = 0$, 
thus giving rise to the $U_q[\mathfrak{gl}(2)]$ symmetry of $\tilde{H}(q,1)$.

We stress that $\comm{h_{L,1}(q,1)}{S^\pm} \neq 0$. One the other hand
$\comm{h_{L,1}(q,\alpha,\beta)}{S^z} = 0$.
Hence for periodic boundary conditions the symmetry 
breaks down to only a residual
$U(1)$ symmetry $\comm{H(q,\alpha,\beta)}{S^z} = 0$ generated by $S^z$,
which corresponds to particle number conservation since the $z$-component of the
total spin
$S^z$ is related to the particle number operator $\hat{N}$ through $S^z = L/2-\hat{N}$.

We also define 
\bel{Spmz}
S^\pm(q,\alpha) = \sum_{k=1}^L S_k^\pm(q,\alpha)
\ee
where
\bel{Sipm}
S_k^\pm(q,\alpha) = \alpha^{\pm \half (L+1-2k)} q^{\half \sum_{i=1}^{k-1} \sigma_i^z
- \half \sum_{i=k+1}^{L} \sigma_i^z} \sigma_k^\pm.
\ee
The diagonal transformation \eref{Vgsig} and the defining relation 
\eref{Uqslncomm2b} yield
\bea
\label{Sgauge}
V(\gamma) S^\pm(q,\alpha) V^{-1}(\gamma) & = & S^\pm \left(q,\alpha \gamma^{-1}\right) \\
\label{Snumber}
W(z) S^\pm(q,\alpha) W^{-1}(z)  & = & z^{\mp 1} S^\pm(q,\alpha).
\eea
Notice that $S^\pm (q,1) = S^\pm$ as defined in \eref{repX}.
Hence $S^\pm (q,\alpha)$ and $S^z$ also form a representation of $U_q[\mathfrak{sl}(2)]$.
Since according to \eref{symmetryH1} 
$\tilde{H}(q,1)$ commutes with the generators $S^\pm=S^\pm (q,1)$ we conclude from
\eref{Hgaugerefl} that
$\tilde{H}(q,\alpha)$ commutes with $S^\pm (q,\alpha)$, which together with $S^z$
form an equivalent representation of $U_q[\mathfrak{sl}(2)]$. In particular,
the generator of the ASEP with reflecting boundary conditions $\tilde{H}=\tilde{H}(q,q)$
commutes with $\tilde{S}^\pm := S^\pm (q,q)$.

We note that
\bea
\label{SpmT}
(S^\pm(q,\alpha))^T & = & S^\mp(q,\alpha^{-1}) \\
\label{reflrep}
\hat{R} S^\pm(q,\alpha) \hat{R}^{-1} & = & S^\pm(q^{-1},\alpha^{-1}) .
\eea
To prove the second equality one uses
$\hat{R} S_k^\pm(q,\alpha) \hat{R}^{-1} = S_{L+1-k}^\pm(q^{-1},\alpha^{-1})$
which comes from \eref{Def:Reflection}.

Finally we introduce the symmetric $q$-number
\bel{qnumber}
[x]_q := \frac{q^x - q^{-x}}{q-q^{-1}}
\ee
for $q,\,q^{-1} \neq 0$ and $x\in\C$. 
This definition can be applied straightforwardly to finite-dimensional matrices 
through the Taylor expansion of the exponential. 
For integers we also define the $q$-factorial
\bel{qfactorial}
[n]_q! := \left\{ \ba{ll} 1 & \quad n=0 \\
\prod_{k=1}^n [k]_q & \quad n \geq 1. \ea \right.
\ee

\subsection{Duality in the quantum Hamiltonian formalism}

For self-containedness we briefly review how to express expectation values
using the matrix representation of the generator which allows to 
state the notion of duality in a neat matrix form \cite{Sudb95,Giar09}.

A probability measure $P(\eta)$ is represented by the column vector
\bel{probvec}
\ket{P} = \sum_{\eta} P(\eta) \ket{\eta}.
\ee
Next we define the {\it summation vector}
\bel{sumvec}
\bra{s} := \sum_{\eta} \bra{\eta}
\ee 
which is the row vector where all components are equal to 1.
The expectation $\exval{f}_P$ of a function $f(\eta)$ 
with respect to a probability distribution $P(\eta)$
is the inner product
\bel{exval}
\exval{f}_P = \inprod{f}{P} = \bra{s} \hat{f} \ket{P}
\ee
where 
\bel{functionmatrix}
\hat{f} :=  \sum_{\eta} f(\eta)  \ket{\eta}  \bra{\eta}
\ee 
is a diagonal matrix
with diagonal elements   $f(\eta)$. 
Notice that
\bel{felement}
f(\eta) = \bra{\eta} \hat{f} \ket{\eta} = \bra{s}  \hat{f} \ket{\eta}.
\ee
One obtains the diagonal matrix $\hat{f}$ \eref{functionmatrix} corresponding to a
function $f(\eta)$ by substituting in
$f(\eta)$ the variable $\eta(k)$ by the diagonal matrix $\hat{n}_k$.

For a Markov process $\eta_t$ the master equation \eref{transgenerator} for
a probability measure $P(\eta_t) := \Prob{\eta_t=\eta)}$ reads
\bel{master}
\frac{\rmd}{\rmd t} \ket{P(t)} = - H \ket{P(t)}
\ee
which implies 
\bel{probvecMarkov}
\ket{P(t)} = \rme^{-Ht} \ket{P_0}
\ee
for an initial probability measure $P_0(\eta)\equiv P(\eta_0)$ at time $t=0$.
We write the expectation of a function $f(\eta_t)$  as 
\bel{exvalMarkov}
\exval{f(t)} := \sum_\eta f(\eta) P(\eta_t) 
= \sum_{\eta} f(\eta) \bra{\eta} \rme^{-Ht} \ket{P_0} = \bra{s} \hat{f} \rme^{-Ht} \ket{P_0}.
\ee
If the initial distribution needs to be specified we use an upper index 
$\exval{f(t)}^{P_0}$.
Normalization implies
$\inprod{s}{P(t)} = 1$ for all $t\geq 0$ and therefore $\bra{s} H = 0$.
A stationary distribution, denoted by $\ket{\pi^\ast}$,
is a right
eigenvector of $H$ with eigenvalue 0, i.e., $H \ket{\pi^\ast} = 0$ 
and normalization $\inprod{s}{\pi^\ast} = 1$.
For the ergodic subspaces with fixed particle number $N$ it is unique.

In order to introduce duality we consider a process $\xi_t$ with generator $H$ and 
a process
$x_t$ with generator $G$ which may have different countable state spaces 
$\Omega_A$ and $\Omega_B$. Consider also a family of functions $f^{x} : \Omega_A \mapsto \C$ 
indexed by $x \in \Omega_B$ and a family of functions $g^{\xi} : \Omega_B \mapsto \C$ 
indexed by $\xi \in \Omega_A$ such that
$f^x(\xi) = g^{\xi}(x) =: D(x,\xi)$. Let the process $\xi_t$ start at some fixed $\xi
\in \Omega_A$
 and let $x_t$ start at some fixed $x\in \Omega_B$. Then the two processes 
are said to be dual with respect to the duality function $D(x,\xi)$ if \cite{Ligg85}
\bel{duality}
\exval{f^{x}(t)}^\xi = \exval{g^\xi(t)}^x.
\ee
As pointed out in \cite{Giar09} this property can be stated neatly in terms 
of the generators as
\bel{duality2}
D H = G^T D
\ee
where the duality matrix $D$ is defined by
\bel{dualitymatrix}
D = \sum_{\xi\in\Omega_A} \sum_{x\in\Omega_B} D(x,\xi) \ket{x}\bra{\xi}
\ee
By construction one has $D(x,\xi) = \bra{x} D \ket{\xi}$.

\subsection{Shock/Antishock measures}

In vector notation a product measure with marginals $\rho_k$ is a tensor vector
$\ket{\{\rho_i\}} = |\rho_1) \otimes \dots \otimes |\rho_L)$ with the single-site column
vectors $|\rho_k) = (1-\rho_k, \rho_k)^T$. It is convenient to introduce the local
fugacity
\bel{fugacity}
z_k = \frac{\rho_k}{1-\rho_k}
\ee
and write the product measure in the form 
$\ket{\{z_i\}} = |z_1) \otimes \dots \otimes |z_L) / Z_L$
with $|z_k) = (1, z_k)^T$ and normalization $Z_L = \prod_{k=1}^L [z_k/(1+z_k)]$.

Specifically, we define 
for a set $\bfx$ of lattice 
sites with cardinality
$K=|\bfx|$ the following family 
$\ket{\nu_{\bfx}} =|z_1) \otimes \dots \otimes |z_L)/Z_L$ of 
shock/antishock measures,
or SAM for short, in terms of the fugacities
\bel{SAMI}
z_k = \left\{
\ba{ll} z q^{2 l} & \mbox{ for } x_l < k < x_{l+1}, \quad l \in \{0, 1, \dots , K+1\} \\
\infty & \mbox{ for }  k \in \bfx \ea \right.
\ee
with $x_0=0$ and $x_{K+1} = L+1$.
On coarse-grained scale with $\xi=k/L$ and $\xi_i^s=x_i/L$ the macroscopic density 
profile $\rho(\xi)$ corresponding to the fugacities $z_k$ has
discontinuities at $\xi=\xi^s_i$ with constant fugacity ratios $z_i^+/z_i^- = q^2$ 
where $z_i^\pm = z_{\xi^s_i\pm 1}$.
In forward (clockwise) direction and for $q>1$ this is an
upward step, corresponding to a shock profile 
for the ASEP (with positive bias $q>1$). Between site $L$ and site 1 there is
downward jump with fugacity ratio $q^{-2K}$.
On macroscopic scale this constitutes an
antishock 
at position $\xi^a = \xi^s - (1+\kappa)/2 \ \mbox{mod} \ 1$, hence the term SAM. 
These shock measures are closely
related to the shock measures defined in \cite{Beli02} and also to the 
infinite-volume shock measures
studied in \cite{Bala10} where the shock positions $x_i$ are occupied by second-class
particles.

With a different normalization factor the general SAM \eref{SAMI} with constant
fugacity jumps $q^2$ can be written as
\bel{SAMIalt}
\ket{\bar{\mu}_{\bfx}} 
:= \prod_{j=1}^K z^{-1}
q^{-\sum_{i=1}^{x_j-1} \hat{n}_i + \sum_{i=x_j+1}^{L} \hat{n}_i} \hat{n}_{x_j}  
\ket{z} \propto 
\ket{\nu_{\bfx}}.
\ee
Here 
\be 
\ket{z} := |z)^{\otimes L}
\ee
for $K=0$ is the {\it unnormalized} 
homogeneous product measure
corresponding to the Bernoulli product measure $\ket{\rho} = \ket{z}/(1+z)^L$ where
$\rho=z/(1+z)$. 

From the SAM defined by 
\eref{SAMIalt} we construct a second type of SAM'S using the transformations
\eref{Vg} and \eref{numb} 
\bea
\label{SAMII}
\ket{\mu_{\bfx}} & := & z^{-K} V(q^{\frac{2K}{L}}) \prod_{j=1}^K W(q^{\frac{2x_j-L-1}{L}})  \ket{\bar{\mu}_{\bfx}} \\
& = &
\label{SAMIIb}
\prod_{j=1}^K \left[
z^{-1} q^{\frac{2}{L}\sum_{l=1}^L (x_j-l)\hat{n}_l
- \sum_{i=1}^{x_j-1} \hat{n}_i +
\sum_{i=x_j+1}^{L} \hat{n}_i} \hat{n}_{x_j} \right] \ket{z}.
\eea
We illustrate the definition for $K=1$ and $K=2$.

For $K=1$ the SAM \eref{SAMII} reduces to
\bel{SAM1}
\ket{\mu_{x}} := 
z^{-1} q^{\frac{2}{L} \sum_{l=1}^L (x-l)\hat{n}_l
- \sum_{i=1}^{x-1} \hat{n}_i + \sum_{i=x+1}^{L} \hat{n}_i} \hat{n}_{x} \ket{z}
\ee
with $1 \leq x \leq L$.
This corresponds to local fugacities
\be 
z_k = \left\{ \ba{ll} zq^{-\frac{2(k-x)+L}{L}} & \mbox{ for } 1 \leq k < x \\ \infty  & \mbox{ for } k = x \\
z q^{-\frac{2(k-x)-L}{L}} & \mbox{ for } x < k \leq L \ea \right.
\ee
and therefore to densities
\be 
\rho_k = \left\{ \ba{ll} 
\frac{1}{2} \left[ 1 - \tanh{\left(\frac{E}{L}(k-x+L\frac{\kappa+1}{2})\right)} \right]
 & \mbox{ for } k < x \\ 1  & \mbox{ for } k = x \\
\frac{1}{2} \left[ 1 - \tanh{\left(\frac{E}{L}(k-x+L\frac{\kappa-1}{2})\right)} \right]
 & \mbox{ for } k > x \ea \right.
\ee
where $E= \ln{q}$ and $\kappa =  \ln{z}/E$
corresponding to $z=q^{\kappa}$. These measures are closely
related to the type-II shock measures defined in \cite{Beli13}. 
On coarse-grained scale with $\xi=k/L$ and $\xi^s=x/L$ the macroscopic density 
profile $\rho(\xi)$ has a
discontinuity at $\xi=\xi^s$ with amplitude
$A^s := \rho^+ - \rho_- = \tanh{(E(\kappa+1)/2)} - \tanh{(E(\kappa+1)/2)}$, 
where $\rho^\pm = \lim_{\epsilon \to 0} \rho(\xi^s \pm \epsilon)$.
In forward (clockwise) direction and for $E>0$ this is an
upward step with fugacity ratio $z^+/z^- = q^2$, corresponding to a shock profile 
for the ASEP (with positive bias $E>0$).
For strong asymmetry $E = \epsilon L$ one has
near $k=x-L(1+\kappa)/2 \ \mbox{mod} \ L$ a smoothened downward ``step'' with an 
intrinsic width 
$\propto 1/\epsilon$ on lattice scale. On macroscopic scale this constitutes an
antishock 
at position $\xi^a = \xi^s - (1+\kappa)/2 \ \mbox{mod} \ 1$. 

For $K=2$ the SAM $\ket{\mu_{x,y}}$ \eref{SAMII} with $1 \leq x < y \leq L$ has 
local fugacities
\be 
z_k = \left\{ \ba{ll} 
z q^{-\frac{2(2k-x-y+L)}{L}} & \mbox{ for } 1 \leq k < x \\
\infty & \mbox{ for }  k = x \\
z q^{-\frac{2(2k-x-y)}{L}} & \mbox{ for } x < k < y \\
\infty & \mbox{ for }  k = y \\
z q^{-\frac{2(2k-x-y-L)}{L}} & \mbox{ for } y < k  \leq L
\ea \right.
\ee
corresponding to densities
\be 
\rho_k = \left\{ \ba{ll} 
\frac{1}{2} \left[ 1 - \tanh{\left(\frac{E}{L}(2k-x-y+L\frac{\kappa+2}{2})\right)} \right]
& \mbox{ for } 1 \leq k < x \\
1 & \mbox{ for }  k = x \\
\frac{1}{2} \left[ 1 - \tanh{\left(\frac{E}{L}(2k-x-y+L\frac{\kappa}{2})\right)} \right]
& \mbox{ for } x < k < y \\
1 & \mbox{ for }  k = y \\
\frac{1}{2} \left[ 1 - \tanh{\left(\frac{E}{L}(2k-x-y+L\frac{\kappa-2}{2})\right)} \right]
& \mbox{ for } y < k  \leq L
\ea \right.
\ee
On macroscopic scale this density profile has two shock discontinuities at 
$\xi^s_{1} = x/L \ \mbox{mod} \ 1$ and $\xi^s_{2} = y/L \ \mbox{mod} \ 1$. Both fugacity
ratios are of magnitude $q^2$. For strong asymmetry $E = \epsilon L$ there
are two antishocks at $\xi^a_1 = (\xi^s_{1} + \xi^s_{2})/2 - (\kappa + 2)/4  \ \mbox{mod} \ 1$
and $\xi^a_2 = (\xi^s_{1} + \xi^s_{2})/2 - (\kappa + 2)/4 \ \mbox{mod} \ 1$.

\section{Results}

Before stating the new results we recall the duality relation for the ASEP with 
reflecting boundary conditions derived in
\cite{Schu97}, Eq. (3.12). We reformulate this duality relation slightly and correct
a sign error in Eq. (3.12) of \cite{Schu97}. We also give a new proof,
parts of which are 
then used to prove the new results given below. We also present a generalized and slightly reformulated version of 
the intertwiner relation Eq. (2.62b) of \cite{Pasq90} for perdiodic  boundary conditions,
also with a correction of some sign errors in that formula.

\begin{theorem} (Sch\"utz, \cite{Schu97})
\label{theo1}
The ASEP with reflecting boundary conditions and asymmetry parameter 
$q$ is self-dual w.r.t. the duality function
\bel{DualityfunctionASEP}
D(\bfx,\eta) =  \prod_{j=1}^{|\bfx|}  q^{ -2 x_j} Q_{x_j} (\eta)
\ee
where
\bel{Q}
Q_{x_j} (\eta) = q^{\sum_{i=1}^{x_j-1} \eta(i) - \sum_{i=x_j+1}^{L} \eta(i) } \eta(x_j).
\ee
\end{theorem}

\begin{remark}
Because of particle number conservation also 
\be 
\tilde{D}(\bfx,\eta) =
q^{|\bfx|(N(\eta)-1)}D(\bfx,\eta) = \prod_{j=1}^{|\bfx|} q^{2N_{x_j}(\eta)-2 x_j}  \eta(x_j)
\ee
is a duality function with the particle numbers $N(\eta)$ \eref{partnum}
and $N_{x}(\eta)$ \eref{NyMx}. This is the duality function (3.12)
of \cite{Schu97}.\footnote{Notice
a sign error in front of the term $2k_i$ in Eq. (3.12) of \cite{Schu97} and pay 
attention to the
different convention $q \leftrightarrow q^{-1}$.}
\end{remark}

\begin{proposition} 
\label{Prop1} 
Let $H(\cdot,\cdot,\cdot)$ be the conditioned generator \eref{ASEPgenper}
of the ASEP with periodic boundary conditions 
and let $\eta_K \in \Omega_K$ be
any configuration with $K$ particles. Then for $0\leq n \leq L-K$
one has the intertwining relation
\be
\label{sym}
\left[(S^\pm(q,\alpha))^n H(q,\alpha,q^{2n}\beta^\pm) -
H(q,\alpha,\beta^\pm) (S^\pm(q,\alpha))^n \right] \ket{\eta_{K}} = 0
\ee
with
\bel{symN0}
\beta^\pm  = q^{\pm (L-2K)} \alpha^{-L}
\ee
and the generators $S^\pm(q,\alpha)$ \eref{Spmz} of $U_q[\mathfrak{gl}(2)]$.
\end{proposition}

\begin{remark}
Defining the duality matrix
$D_{K}^{K\pm n} = \mathbf{1}_{K\mp n} (S^\pm(q,\alpha))^n \mathbf{1}_{K}$
with the projector \eref{projectorN} and using \eref{HTper}
the intertwiner relation \eref{sym} can be expressed as the duality relation
\bel{dualityrelation}
D_{K}^{K\mp n} H_K(q,\alpha,q^{2n}\beta^\pm) = 
\left(H_{K\mp n} (q,\alpha^{-1},\beta^\mp)\right)^T D_{K}^{K\mp n}
\ee
with the projected generator \eref{generatorN}. We shall focus on the formulation
\eref{sym} of this duality.
\end{remark}

\begin{remark} For $\alpha=1$ this is the result (2.62b) of \cite{Pasq90}.\footnote{Eqs. 
(2.62a) and (2.62b) of \cite{Pasq90} have some sign errors which are corrected in 
Proposition \eref{Prop1}.} The proof of Proposition \eref{Prop1} is entirely analogous
to the derivation given in \cite{Pasq90} since the generalized form \eref{sym} follows 
trivially from the result of \cite{Pasq90} for $\alpha=1$ through the similarity transformation \eref{Vg}. 
Some ingredients of the proof, with sign errors in \cite{Pasq90} corrected, 
are presented in the appendix. 
\end{remark}

We focus now on global conditioning ($\alpha\neq q, \ \beta=1$) and local conditioning $\alpha=q, \ \beta \neq 1$. The main results of this work are the following theorems.

\begin{theorem} 
\label{theo2}
Let $H^K_N := \mathbf{1}_N H(q,q^{1-\frac{2K}{L}},1)\mathbf{1}_N$  
be the generator \eref{generatorN} of the globally conditioned ASEP with $N$ particles and
periodic boundary conditions and driving strength $s = -2K/L \ln{q}$. Furthermore, let 
$\ket{\mu^N_{\bfx}} = \mathbf{1}_N \ket{\mu_{\bfx}}$ be the unnormalized shock-antishock measure \eref{SAMII} 
restricted to $N$ particles and 
\be 
\ket{\mu^N_{\bfx}(t)} := \rme^{-H^K_N t} \ket{\mu^N_{\bfx}}
\ee
 with $K=|\bfx|$. Then
\bel{shocks}
\ket{\mu^N_{\bfx}(t)} = \sum_{\bfy\in\Omega_K} P^N(\bfy,t|\bfx,0) \\
\ket{\mu^N_{\bfy}}
\ee
where $P^N(\bfy,t|\bfx,0) := \bra{\bfy} \rme^{-H^N_K t} \ket{\bfx}$ is the
conditioned $K$-particle transition probability from $\bfx$ to $\bfy$ at time $t$
with driving strength $s' = -2N/L \ln{q}$.
\end{theorem}

\begin{remark}
The significance of this result lies in the fact that the conditioned evolution of an
$N$-particle SAM is fully determined by the conditioned transition probability of only 
$K$ particles, in analogy to the evolution of shocks in the infinite lattice
explored in \cite{Beli02,Bala10}. 
\end{remark}

\begin{remark}
For $K=1$ a related result was obtained in \cite{Beli13} for a normalized and
slightly different definition of the shock measures. The proof of \cite{Beli13} is by explicit computations relying on the presence of a single shock.
The present proof for the generalized $K\geq 1$ shows that the mathematical origin
of the conditioned shock motion is the duality relation  \eref{sym}.
\end{remark}

\begin{theorem} 
\label{theo3}
Let $\bar{H}^K_N := \mathbf{1}_N H(q,q,q^{-2K})\mathbf{1}_N$  
be the generator \eref{generatorN} of the locally conditioned ASEP with $N$ particles and
periodic boundary conditions and boundary driving strength $\bar{s} = -2K \ln{q}$. Furthermore, let 
$\ket{\bar{\mu}^{N}_{\bfx}}= \mathbf{1}_N \ket{\bar{\mu}_{\bfx}}$ be the unnormalized shock-antishock measure \eref{SAMI}
restricted to $N$ particles and 
\be 
\ket{\bar{\mu}^{N}_{\bfx}(t)} := \rme^{-\bar{H}^K_N t} \ket{\bar{\mu}^{N}_{\bfx}}
\ee
with $K=|\bfx|$. Then
\bel{shocks2}
\ket{\bar{\mu}^{N}_{\bfx}(t)} = \sum_{\bfy\in\Omega_K} \bar{P}^N(\bfy,t|\bfx,0) \\
\ket{\bar{\mu}^{N}_{\bfy}}
\ee
where $\bar{P}^N(\bfy,t|\bfx,0) := \bra{\bfy} \rme^{-\bar{H}^N_K t} \ket{\bfx}$ is the
boundary-conditioned $K$-particle transition probability from $\bfx$ to $\bfy$ at time $t$
with driving strength $\bar{s}' = -2N \ln{q}$.
\end{theorem}
\section{Proofs}

\subsection{Proof of Theorem \ref{theo1}}

\begin{proof}

We first note
\begin{lemma}
\label{Lem1}
Let 
\be
\tilde{S} = \sum_{n=0}^L \frac{\tilde{S}^+}{[n]_q!}, \quad \hat{Q}_x = 
q^{\sum_{i=1}^{x-1} \hat{n}_i - \sum_{i=x+1}^{L} \hat{n}_i } \hat{n}_{x}.
\ee
Then for a configuration $\bfx \in \Omega_{N}$ with $N=|\bfx|$ particles one has
\bel{aux1}
\bra{\bfx} \tilde{S} = \bra{s} \prod_{i=1}^{|\bfx|} \hat{Q}_{x_i}.
\ee
\end{lemma}

The proof is completely analogous to the proof in \cite{Beli15b} 
of (164) 
with $\bfy=\emptyset$.

Now we observe that with the reversible measure \eref{ASEPrevmeas} and with \eref{felement} 
we can write
\bel{aux2}
D(\bfx,\eta) = \pi^{-1}(\bfx) \bra{s} \prod_{i=1}^{|\bfx|} \hat{Q}_{x_i} \ket{\eta} 
=f^{\bfx}(\eta)=g^\eta(\bfx).
\ee
Then the following chain of equalities holds and proves the theorem:
\bea
\label{eq1}
\exval{f^{\bfx}(t)}^\eta & := & \sum_\xi f^{\bfx}(\xi) \bra{\xi} \rme^{-\tilde{H}t}  \ket{\eta}  \\
\label{eq2}
& = & \sum_\xi \pi^{-1}(\bfx) \bra{s} \prod_{i=1}^{|\bfx|} \hat{Q}_{x_i}  \ket{\xi}
\bra{\xi} \rme^{-\tilde{H}t}  \ket{\eta}  \\
\label{eq3}
& = & \pi^{-1}(\bfx) \bra{s} \prod_{i=1}^{|\bfx|} \hat{Q}_{x_i}  
\rme^{-\tilde{H}t}  \ket{\eta}  \\
\label{eq4}
& = & \pi^{-1}(\bfx) \bra{\bfx} \tilde{S} \rme^{-\tilde{H}t}  \ket{\eta}  \\
\label{eq5}
& = & \pi^{-1}(\bfx) \bra{\bfx} \rme^{-\tilde{H}t} \tilde{S}   \ket{\eta}  \\
\label{eq6}
& = & \pi^{-1}(\bfx) \sum_{\bfy} \bra{\bfx}  
\rme^{-\tilde{H}t} \ket{\bfy}\bra{\bfy} \tilde{S}   \ket{\eta}  \\
\label{eq7}
& = &  \sum_{\bfy\in\Omega_N} \bra{\bfx} \hat{\pi}^{-1}
\rme^{-\tilde{H}t} \hat{\pi} \ket{\bfy} 
\pi^{-1}(\bfy) \bra{s} \prod_{i=1}^{|\bfy|} \hat{Q}_{y_i}  \ket{\eta}  \\
\label{eq8}
& = &  \sum_{\bfy\in\Omega_N} \bra{\bfy} 
\rme^{-\tilde{H}t} \ket{\bfx} 
\pi^{-1}(\bfy) \bra{s} \prod_{i=1}^{|\bfy|} \hat{Q}_{y_i}  \ket{\eta}  \\
\label{eq9}
& = &  \sum_{\bfy\in\Omega_N} g^\eta(\bfy) \bra{\bfy} 
\rme^{-\tilde{H}t} \ket{\bfx}  \\
\label{eq10}
& =: & \exval{g^{\eta}(t)}^{\bfx}
\eea
The following ingredients were used: Eqs. \eref{eq1} and \eref{eq10}: The expressions \eref{exvalMarkov}
for expectations; Eqs.  \eref{eq2} and \eref{eq9}: The expression \eref{aux2} for the duality function; Eq. \eref{eq3}: The expressions \eref{exvalMarkov} and the representation
\eref{unit} of the unit matrix of dimension $2^L$; Eq.
\eref{eq4}: The expression \eref{aux1} of part of the duality function 
in terms of the symmetry operator $\tilde{S}$; Eq. \eref{eq5}: The 
$U_q[\mathfrak{sl}(2)]$symmetry \eref{symmetryH1}; Eq. \eref{eq6}: Particle number 
conservation and the representation
\eref{projectorN} of the unit matrix in the subspace of $N$ particles; Eq. 
\eref{eq7}: The diagonal matrix representation of the reversible measure 
\eref{ASEPrevmeas}; Eq. \eref{eq8}: Reversibility \eref{revASEPrefl}.
\hfill \qed
\end{proof}

\subsection{Proofs of Theorems \ref{theo2} and \ref{theo3}}

Before we set out to prove Theorem \ref{theo2} and Theorem \ref{theo3}
we show that the SAM \eref{SAMII}
can be generated by the action of the particle creation operator $U_q[\mathfrak{sl}(2)]$.

\begin{lemma}
\label{Lem3}
Let $\bfx$ be a configuration of $K=|\bfx|$ particles and let 
\be
\bar{\mu}^N_{\bfx} := \bar{\mu}_{\bfx} \delta_{\sum_{k=1}^L \eta(k), N}, \quad
\mu^N_{\bfx} := \mu_{\bfx} \delta_{\sum_{k=1}^L \eta(k), N}
\ee
be the SAM's defined by \eref{SAMI}, \eref{SAMII} restricted to $N \geq K$ particles. Then the
vector representations $\ket{\bar{\mu}^N_{\bfx}}
:= \mathbf{1}_N \ket{\bar{\mu}_{\bfx}}$ and
$\ket{\mu^N_{\bfx}}
:= \mathbf{1}_N \ket{\mu_{\bfx}}$ can be written as
\bea
\label{SAMUqI}
\ket{\bar{\mu}^N_{\bfx}} & = & z^{N-K}
\frac{\left(S^-(q^{-1},q)\right)^{N-K}}{[N-K]_q!} \ket{\bfx} \\
\label{SAMUqII}
\ket{\mu^N_{\bfx}} & = & z^{N-K} V(q^{\frac{2(K-N)}{L}}) \frac{\left(S^-(q^{-1},q^{1-\frac{N}{L}})\right)^{N-K}}{[N-K]_q!} \ket{\bfx}
\eea
in terms of the generators \eref{Spmz} of $U_q[\mathfrak{sl}(2)]$ and the transformation
\eref{Vg}.
\end{lemma}

\begin{proof}
From Lemma \ref{Lem1} and \eref{SpmT} one finds
\bel{p1}
\sum_{n=0}^{L-K} \frac{\left(S^-(q^{-1},q)\right)^{n}}{[n]_q!} \ket{\bfx} 
= \prod_{j=1}^K 
q^{-\sum_{i=1}^{x_j-1} \hat{n}_i + \sum_{i=x_j+1}^{L} \hat{n}_i} \hat{n}_{x_j}  
\ket{s}.
\ee
Notice that for $z=1$ one has $\ket{z=1}=\ket{s}$.

The transformation \eref{Vg} yields $V(\gamma) \ket{\bfx} = \gamma^{-\frac{1}{2}
\sum_{j=1}^K (2 x_j-L-1)} \ket{\bfx}$ and \eref{Sgauge} gives
$S^{-}(q^{-1},q) = V^{-1}(\lambda) S^{-}(q^{-1},q\lambda^{-1}) V(\lambda)$.
Putting this together and using \eref{Snumber} turns \eref{p1} into
\bea
\label{p2}
& & V(\gamma) \sum_{n=0}^{L-K} \frac{z^n \left(S^-(q^{-1},q\lambda^{-1})\right)^{n}}{[n]_q!} \ket{\bfx} \nonumber \\
& & = V(\gamma\lambda) \prod_{j=1}^K z^{-1} \lambda^{\half (2x_j-L-1)}
q^{-\sum_{i=1}^{x_j-1} \hat{n}_i + \sum_{i=x_j+1}^{L} \hat{n}_i} \hat{n}_{x_j}  
\ket{z}.
\eea

Now we choose $\lambda = q^{\frac{2N}{L}}$ and $\gamma=q^{\frac{2(K-N)}{L}}$ to obtain
\bea
\label{p3}
& & 
V(q^{\frac{2(K-N)}{L}})
\sum_{n=0}^{L-K} \frac{z^n \left(S^-(q^{-1},q^{1-\frac{2N}{L}})\right)^{n}}{[n]_q!} \ket{\bfx} 
\nonumber \\
& & = \prod_{j=1}^K z^{-1} 
q^{\frac{N}{L} (2x_j-L-1)-\frac{1}{L}\sum_{l=1}^L(2l-L-1)\hat{n}_l }
q^{-\sum_{i=1}^{x_j-1} \hat{n}_i + \sum_{i=x_j+1}^{L} \hat{n}_i} \hat{n}_{x_j}  
\ket{z}.
\eea

Finally one applies the projector $\mathbf{1}_N$ on both sides of the equation. On the
l.h.s. this projects out the term with $n=N-K$, corresponding to the r.h.s. 
of \eref{SAMUqII}.
On the r.h.s. the projection allows us to substitute the number $N$ in the first power
of $q$ by the number operator $\hat{N}$ \eref{numb}. Thus the terms proportional to $L+1$
cancel and the expression \eref{SAMIIb} remains under the projection operator. 
Therefore the r.h.s. is equal to $\ket{\mu^N_{\bfx}}$. Similarly one chooses
$\gamma=\lambda=1$ to obtain \eref{SAMUqI}. \qed
\end{proof}

\subsubsection{Proof of Theorem \ref{theo2}}

Now we are in a position to prove \eref{shocks}.

Consider a $K$-particle configuration $\bfx$
and the duality relation \eref{sym} with $n=N-K$ and $\beta=1$: 
\bea
& & \frac{(S^-(q,q^{2\frac{K}{L}-1}))^{N-K}}{[N-K]_q!} 
H(q,q^{2\frac{K}{L}-1},q^{2N-2K}) \ket{\bfx} \nonumber \\
\label{symb1a}
& & \quad =
H(q,q^{2\frac{K}{L}-1},1) 
\frac{(S^-(q,q^{2\frac{K}{L}-1}))^{N-K}}{[N-K]_q!} \ket{\bfx}.
\eea
With the transformation \eref{Vg} with 
$\gamma^L=q^{2K-2N}$
one uses \eref{Hgaugeper} to cast this in the form
\bea
& & 
H(q,q^{2\frac{K}{L}-1},1)
\frac{(S^-(q,q^{2\frac{K}{L}-1}))^{N-K}}{[N-K]_q!}
\ket{\bfx}  \nonumber \\
\label{symb1b}
& & \quad =  
\frac{(S^-(q,q^{2\frac{K}{L}-1}))^{N-K}}{[N-K]_q!} V^{-1}(\gamma)  H(q,q^{2\frac{N}{L}-1},1)  
V(\gamma) \ket{\bfx}
\eea
or, alternatively,
\bea
& & 
H(q,q^{2\frac{K}{L}-1},1) V^{-1}(\gamma) 
\frac{(S^-(q,q^{2\frac{N}{L}-1}))^{N-K}}{[N-K]_q!} 
\ket{\bfx}  \nonumber \\
\label{symb1c}
& & \quad =  V^{-1}(\gamma) 
\frac{(S^-(q,q^{2\frac{N}{L}-1}))^{N-K}}{[N-K]_q!}  
H(q,q^{2\frac{N}{L}-1},1)  
\ket{\bfx}.
\eea
Applying \eref{symgen1}, \eref{reflrep}, \eref{Vgsym} this turns into
\bea
& & 
H(q,q^{1-2\frac{K}{L}},1) V(\gamma) 
\frac{(S^-(q^{-1},q^{1-2\frac{N}{L}}))^{N-K}}{[N-K]_q!} 
\ket{\tilde{\bfx}}  \nonumber \\
\label{symb1d}
& & \quad =  V(\gamma) 
\frac{(S^-(q^{-1},q^{1-2\frac{N}{L}}))^{N-K}}{[N-K]_q!}  
H(q,q^{1-2\frac{N}{L}},1)  
\ket{\tilde{\bfx}}.
\eea
where $\ket{\tilde{\bfx}} = V^{-1}(\gamma) \ket{\bfx}$ is an arbitrary $K$-particle
configuration.

Since $H$ conserves particle number, this relation remains valid for any
power of $H$. Thus we find
\bea
& & 
\rme^{-H^{K}_N t} V(\gamma) 
\frac{(S^-(q^{-1},q^{1-2\frac{N}{L}}))^{N-K}}{[N-K]_q!} 
\ket{\tilde{\bfx}}  \nonumber \\
\label{symb1e}
& & \quad =  V(\gamma) 
\frac{(S^-(q^{-1},q^{1-2\frac{N}{L}}))^{N-K}}{[N-K]_q!}  
\rme^{-H^N_{K} t}
\ket{\tilde{\bfx}} \\
& & \quad =  \sum_{\eta'_{K}} V(\gamma) 
\frac{(S^-(q^{-1},q^{1-2\frac{N}{L}}))^{N-K}}{[N-K]_q!} \ket{\eta'_{K}}
\bra{\eta'_{K}} \rme^{-H^N_{K} t} \ket{\tilde{\bfx}}
\eea
where in the last equality we have inserted the unit operator restricted to
$K$-particle states.
Using \eref{SAMUqII} of Lemma \ref{Lem3} then proves \eref{shocks}. \qed

\subsubsection{Proof of Theorem \ref{theo3}}

The proof of Theorem \ref{theo3} is similar. For $\alpha=q^{-1}$ where $\beta = q^{2K}$ one has
\be
\label{sym5}
\frac{(S^-(q,q^{-1}))^{N-K}}{[N-K]_q!} 
H(q,q^{-1},q^{2N}) \ket{\bfx} =
H(q,q^{-1},q^{2K}) 
\frac{(S^-(q,q^{-1}))^{N-K}}{[N-K]_q!} \ket{\bfx}.
\ee
Applying space reflection \eref{symgen1}, \eref{reflrep} this becomes
\be\label{sym6}
\frac{(S^-(q^{-1},q))^{N-K}}{[N-K]_q!} 
H(q,q,q^{-2N}) \ket{\bfx} 
=
H(q,q,q^{-2K}) 
\frac{(S^-(q^{-1},q))^{N-K}}{[N-K]_q!} \ket{\bfx}.
\ee
Here we dropped the tilde over the configuration $\bfx$ since it is arbitrary.

Projecting on $N$ particles and iterating this duality over powers of 
$\bar{H}^K_N$ yields
\be
\label{sym7}
\rme^{-\bar{H}_N^K t} \frac{(S^-(q^{-1},q))^{N-K}}{[N-K]_q!} 
\ket{\bfx} 
= \frac{(S^-(q^{-1},q))^{N-K}}{[N-K]_q!} 
\rme^{-\bar{H}_K^N t} \ket{\bfx} .
\ee
Inserting the unit operator restricted to
$K$-particle states and
Using \eref{SAMUqI} of Lemma \ref{Lem3} then proves \eref{shocks2}. \qed

\begin{acknowledgement}
GMS thanks DFG and FAPESP for financial support.
\end{acknowledgement}

\section*{Appendix}
\addcontentsline{toc}{section}{Appendix}

We present some details of the proof of Proposition \eref{Prop1} 
which are not shown in \cite{Pasq90} 
and from which Proposition \eref{Prop1} follows by the similarity transformation \eref{Vg}.

We define $e_L(\cdot,\cdot,\cdot) := h_{L,1}(\cdot,\cdot,\cdot)$, see \eref{hoppingbound}.
By explizit matrix multiplications one finds from the relations \eref{projhat} for
the bulk operators
\bea
S_k^\pm(q,\alpha) e_L(\alpha',q',\beta) & = &
e_L(\alpha',q',\beta q^{-2}) S_k^\pm(q,\alpha) \quad 2 \leq k \leq L-1 \\
e_L(\alpha',q',\beta) S_k^\pm(q,\alpha) & = & S_k^\pm(q,\alpha) 
e_L(\alpha',q',\beta q^{2})  \quad 2 \leq k \leq L-1
\eea
and for the boundary operators
\bea 
S_1^+(q,\alpha) e_L(\alpha',q',\beta) & = & q^{-1/2} \alpha^{\half(L-1)} 
\left[ (q')^{-1} \sigma^+_1 \hat{\upsilon}_L - \alpha' \beta \hat{\upsilon}_1 \sigma^+_L \right]
q^{-S^z}\\
S_1^-(q,\alpha) e_L(\alpha',q',\beta) & = & q^{1/2} \alpha^{-\half(L-1)} 
\left[ q' \sigma^-_1 \hat{n}_L - (\alpha' \beta)^{-1} \hat{n}_1 \sigma^-_L \right]
q^{-S^z}\\
S_L^+(q,\alpha) e_L(\alpha',q',\beta) & = & q^{1/2} \alpha^{-\half(L-1)} 
\left[ q' \hat{\upsilon}_1 \sigma^+_L - (\alpha' \beta)^{-1} \sigma^+_1 \hat{\upsilon}_L \right]
q^{S^z}\\
S_L^-(q,\alpha) e_L(\alpha',q',\beta) & = & q^{-1/2} \alpha^{\half(L-1)} 
\left[ (q')^{-1} \hat{n}_1 \sigma^-_L - \alpha' \beta \sigma^-_1 \hat{n}_L \right]
q^{S^z}
\eea
and
\bea 
e_L(\alpha',q',\beta) S_1^+(q,\alpha) & = & q^{-1/2} \alpha^{\half(L-1)} 
\left[ q' \sigma^+_1 \hat{n}_L - \alpha' \beta \hat{n}_1 \sigma^+_L \right]
q^{-S^z}\\
e_L(\alpha',q',\beta) S_1^-(q,\alpha) & = & q^{1/2} \alpha^{-\half(L-1)} 
\left[ (q')^{-1} \sigma^-_1 \hat{\upsilon}_L - (\alpha' \beta)^{-1} \hat{\upsilon}_1 \sigma^-_L \right]
q^{-S^z}\\
e_L(\alpha',q',\beta) S_L^+(q,\alpha) & = & q^{1/2} \alpha^{-\half(L-1)} 
\left[ (q')^{-1} \hat{n}_1 \sigma^+_L - (\alpha' \beta)^{-1} \sigma^+_1 \hat{n}_L \right]
q^{S^z}\\
e_L(\alpha',q',\beta) S_L^-(q,\alpha) & = & q^{-1/2} \alpha^{\half(L-1)} 
\left[ q' \hat{\upsilon}_1 \sigma^-_L - \alpha' \beta \sigma^-_1 \hat{\upsilon}_L \right]
q^{S^z}.
\eea

Consider now $q=q'$ and $\alpha=\alpha'$. From the quantum algebra symmetry and
from the previous relations one obtains (omitting the $q,\alpha$-dependence)
\bea
S^\pm H(\beta) - H(\beta') S^\pm & = & S^\pm e_L(\beta) - e_L(\beta') S^\pm \\
& = & \left[ e_L(q^{-2}\beta) - e_L(\beta') \right] \sum_{k=2}^{L-1} S^\pm_k \nonumber\\
\label{pseudocomm}
& & + \left( S^\pm_1 + S^\pm_L \right) e_L(\beta) - e_L(\beta') \left( S^\pm_1 + S^\pm_L \right).
\eea

Observe that
\bea 
\label{Spm1}
& & S_1^+ = q^{-1/2} \alpha^{\half(L-1)} \sigma^+_1 q^{-S^z}, \quad
S_1^- = q^{1/2} \alpha^{-\half(L-1)} \sigma^-_1 q^{-S^z}, \\
\label{SpmL}
& & S_L^+ = q^{1/2} \alpha^{-\half(L-1)} \sigma^+_L q^{S^z}, \quad
S_L^- = q^{-1/2} \alpha^{\half(L-1)} \sigma^-_L q^{S^z}
\eea
and the auxiliary relations
\bea 
\sigma^+_1 e_L(\beta) & = & 
q^{-1} \sigma^+_1 \hat{\upsilon}_L - \alpha\beta \hat{\upsilon}_1 \sigma^+_L, \\
e_L(\beta) \sigma^+_1 & = & 
q \sigma^+_1 \hat{n}_L - \alpha\beta \hat{n}_1 \sigma^+_L \\
\sigma^+_L e_L(\beta) & = &
q \hat{\upsilon}_1 \sigma^+_L - (\alpha\beta)^{-1} \sigma^+_1 \hat{\upsilon}_L, \\
e_L(\beta) \sigma^+_L & = & 
q^{-1} \hat{n}_1 \sigma^+_L - (\alpha\beta)^{-1} \sigma^+_1 \hat{n}_L,
\eea
and
\bea 
\sigma^-_1 e_L(\beta) & = & 
q \sigma^-_1 \hat{n}_L - (\alpha\beta)^{-1} \hat{n}_1 \sigma^-_L, \\
e_L(\beta) \sigma^-_1 & = & 
q^{-1} \sigma^-_1 \hat{\upsilon}_L - (\alpha\beta)^{-1} \hat{\upsilon}_1 \sigma^-_L \\
\sigma^-_L e_L(\beta) & = &
q^{-1} \hat{n}_1 \sigma^-_L - \alpha\beta \sigma^-_1 \hat{n}_L, \\
e_L(\beta) \sigma^-_L & = & 
q \hat{\upsilon}_1 \sigma^-_L - \alpha\beta \sigma^-_1 \hat{\upsilon}_L.
\eea

Thus one obtains
\bea 
& & \left( S^+_1 + S^+_L \right) e_L(\beta) 
- e_L(\beta') \left( S^+_1 + S^+_L \right) \nonumber \\
& & = A^+(\beta,\beta') \beta^{1/2} \alpha^{L/2} q^{-S^z-1} +
B^+(\beta,\beta') \beta^{-1/2} \alpha^{-L/2} q^{S^z+1}
\eea
with
\bea 
A^+(\beta,\beta') & = &
\frac{q^{1/2}}{(\alpha\beta)^{1/2}}\sigma^+_1 
\left[q^{-1} \hat{\upsilon}_L - q \hat{n}_L\right] - 
\frac{(\alpha\beta)^{1/2}}{q^{1/2}} \sigma^+_L
\left[q \hat{\upsilon}_1 - q \frac{\beta'}{\beta}\hat{n}_1\right] \\
B^+(\beta,\beta') & = &
\frac{q^{1/2}}{(\alpha\beta)^{1/2}}\sigma^+_1 
\left[q^{-1} \frac{\beta}{\beta'}\hat{n}_L - q^{-1} \hat{\upsilon}_L\right] - 
\frac{(\alpha\beta)^{1/2}}{q^{1/2}} \sigma^+_L
\left[q^{-1} \hat{n}_1 - q \hat{\upsilon}_1\right] 
\eea
and
\bea 
& & \left( S^-_1 + S^-_L \right) e_L(\beta) 
- e_L(\beta') \left( S^-_1 + S^-_L \right) \nonumber \\
& & = A^-(\beta,\beta') \beta^{-1/2} \alpha^{-L/2} q^{-S^z+1} +
B^-(\beta,\beta') \beta^{1/2} \alpha^{L/2} q^{S^z-1}
\eea
with
\bea 
A^-(\beta,\beta') & = &
\frac{(\alpha\beta)^{1/2}}{q^{1/2}}\sigma^-_1 
\left[q \hat{n}_L - q^{-1} \hat{\upsilon}_L\right] - 
\frac{q^{1/2}}{(\alpha\beta)^{1/2}} \sigma^-_L
\left[q^{-1} \hat{n}_1 - q^{-1} \frac{\beta}{\beta'}\hat{\upsilon}_1\right] \\
B^-(\beta,\beta') & = &
\frac{(\alpha\beta)^{1/2}}{q^{1/2}}\sigma^-_1 
\left[q \frac{\beta'}{\beta}\hat{\upsilon}_L - q \hat{n}_L\right] - 
\frac{q^{1/2}}{(\alpha\beta)^{1/2}} \sigma^-_L
\left[q \hat{\upsilon}_1 - q^{-1} \hat{n}_1\right] .
\eea

With the choice $\beta' = q^{-2} \beta$ \eref{pseudocomm} reduces to
\bel{pseudocomm2}
S^\pm H(\beta) - H(q^{-2} \beta) S^\pm =
\left( S^\pm_1 + S^\pm_L \right) e_L(\beta) - e_L(q^{-2}\beta) 
\left( S^\pm_1 + S^\pm_L \right).
\ee
For $S^+$ the r.h.s. reduces to
\bea
& & \left\{\frac{q^{1/2}}{(\alpha\beta)^{1/2}}\sigma^+_1 
\left[q^{-1} \hat{\upsilon}_L - q \hat{n}_L\right] -
\frac{(\alpha\beta)^{1/2}}{q^{1/2}} \sigma^+_L
\left[q \hat{\upsilon}_1 - q^{-1} \hat{n}_1\right]\right\} \nonumber \\
& & \times \left[ \beta^{1/2} \alpha^{L/2} q^{-S^z-1}
- \beta^{-1/2} \alpha^{-L/2} q^{S^z+1} \right] \nonumber
\eea
With \eref{Spm1}, \eref{SpmL} one thus arrives at
\bea
S^+ H(\beta) - H(q^{-2} \beta) S^+ & = &
\left[q^{-1} \hat{\upsilon}_L - q \hat{n}_L\right] S^+_1 
\left[1-\beta^{-1}\alpha^{-L} q^{2S^z+2} \right] \nonumber \\
\label{pseudocommp}
& & + \left[q \hat{\upsilon}_1 - q^{-1} \hat{n}_1\right] S^+_L 
\left[1-\beta \alpha^{L} q^{-2S^z-2} \right].
\eea
Notice that the action of the pseudo commutator on states with 
particle number satisfying
\be 
q^{L-2N+2} = \beta \alpha^L
\ee
vanishes.

Similarly one obtains for $S^-$ the r.h.s. of \eref{pseudocomm2}
\bea
& & \left\{\frac{(\alpha\beta)^{1/2}}{q^{1/2}}\sigma^-_1 
\left[q^{-1} \hat{\upsilon}_L - q \hat{n}_L\right] -
\frac{q^{1/2}}{(\alpha\beta)^{1/2}} \sigma^-_L
\left[q \hat{\upsilon}_1 - q^{-1} \hat{n}_1\right]\right\} \nonumber \\
& & \times \left[ \beta^{1/2} \alpha^{L/2} q^{S^z-1}
- \beta^{-1/2} \alpha^{-L/2} q^{-S^z+1} \right] \nonumber
\eea
which yields
\bea
S^- H(\beta) - H(q^{-2} \beta) S^- & = &
-\left[q^{-1} \hat{\upsilon}_L - q \hat{n}_L\right] S^-_1 
\left[1-\beta\alpha^{L} q^{2S^z-2} \right] \nonumber \\
\label{pseudocommm}& & - \left[q \hat{\upsilon}_1 - q^{-1} \hat{n}_1\right] S^-_L 
\left[1-\beta^{-1} \alpha^{-L} q^{-2S^z+2} \right].
\eea
Notice that the action of the pseudo commutator on states with 
particle number satisfying
\be 
q^{-L+2N+2} = \beta \alpha^L
\ee
vanishes.

In compact form \eref{pseudocomm2} can thus be written
\bea
S^\pm H(\beta) - H(q^{-2} \beta) S^\pm & = &
\pm \left[q^{-1} \hat{\upsilon}_L - q \hat{n}_L\right] S^\pm_1 
\left[1-\beta^{\mp 1}\alpha^{\mp L} q^{2S^z\pm 2} \right] \nonumber \\
\label{pseudocomm3}
& & \pm \left[q \hat{\upsilon}_1 - q^{-1} \hat{n}_1\right] S^\pm_L 
\left[1-\beta^{\pm 1} \alpha^{\pm L} q^{-2S^z\mp 2} \right].
\eea

One can iterate. E.g. for $(S^-)^2$ one obtains
\bea
& & (S^-)^2 H(\beta) - H(q^{-4} \beta) (S^-)^2 \nonumber \\
& & \quad = (1+q^{-2}) \left[q \hat{n}_L - q^{-1} \hat{\upsilon}_L\right] 
S^-_1 \left(\sum_{k=2}^{L-1} S^-_k\right) 
\left[1-\beta\alpha^{ L} q^{2S^z -4} \right] \nonumber \\
\label{pseudocommm2}
& & \quad + (1+q^{-2})\left[q^{-1} \hat{n}_1 - q \hat{\upsilon}_1\right] 
\left(\sum_{k=2}^{L-1} S^-_k\right) S^-_L 
\left[1-\beta^{- 1} \alpha^{- L} q^{-2S^z+4} \right].
\eea
Iterating further as in \cite{Pasq90} one arrives at Proposition \ref{Prop1}.

\end{document}